\documentclass[]{article}
\usepackage[english]{babel}

\usepackage[a4paper,top=3cm,bottom=3cm,left=3cm,right=3cm,marginparwidth
=1.75cm]{geometry}
\usepackage{indentfirst}
\usepackage{amsmath,amssymb,amsbsy}
\usepackage{mathrsfs}
\usepackage{amsthm}
\usepackage{graphicx}
\usepackage[colorlinks=true, allcolors=blue]{hyperref}
\usepackage{amsfonts}
\usepackage{authblk}
\usepackage{fancyhdr}
\theoremstyle{definition}

\theoremstyle{plain}
\newtheorem{Th}{Theorem}
\newtheorem{Lem}{Lemma}
\theoremstyle{remark}
\newtheorem{rem}{Remark}

\title{Higher moments of the symmetric square $L$-function off the critical line}
\author[1]{Youjun Wang}
\date{}
\begin{document}
	
	\maketitle
	
	\begin{abstract}
		Let $f$ be the Hecke eigenform for the modular group $SL_2(\mathbb{Z})$, and
		$L(s, \textup{sym}^2 f)$
		be the symmetric square $L$-function associated with $f$. For $\frac{1}{2}<\sigma<1$, define
		$m(\sigma)$ as the supremum of all numbers $m$ such that
		\[
		\int_{1}^T|L(\sigma+it, \textup{sym}^2 f)|^m \textup{d}t\ll_f T^{1+\varepsilon},
		\]
		where $\epsilon>0$ is an arbitrarily small number. In this paper, we established the bound
		\begin{align*}
			m(\sigma)\geq \frac{17}{26-28\sigma}, \text{ for }\frac{5}{8}\leq\sigma\leq\frac{52}{73},
		\end{align*}
		which improved our previous result.

		
	\end{abstract}
	
	\section{Introduction}
Let $f(z)$ be a Hecke eigen cusp form of even integral weight $k$ for $\Gamma = SL(2, \mathbb{Z})$. Suppose that $f(z)$ has the following Fourier expansion at the cusp $\infty$:
\[
    f(z) = \sum_{n=1}^{\infty} \lambda_f(n) n^{\frac{k-1}{2}} e(n z),
    \]

where $e(x) := e^{2\pi i x}$ and the $n$-th normalized Fourier coefficient $\lambda_f(n)$ of $f$ is the eigenvalue of the Hecke operator $T_n$. From the theory of Hecke operators, $\lambda_f(n)$ is real, and for all $n \geq 1$,

\begin{align}\label{1.1}
        \vert\lambda_f(n)\vert \leq d(n),
    \end{align}

where $d(n)$ denotes the divisor function. This is the well-known Ramanujan Conjecture, which was proved by Deligne \cite{Deligne} in 1974. In other words, Deligne showed that for any prime $p$, there exist complex numbers $\alpha(p)$ and $\beta(p)$ such that

\[
    \alpha(p) + \beta(p) = \lambda_{f}(p),\quad \alpha(p)\beta(p) = 1 \quad \text{and} \quad \vert\alpha(p)\vert = \vert\beta(p)\vert = 1.
    \]

Therefore, $L(s,f)$ can be expressed as

\[
    L(s,f) = \prod_{p} \left(1 - \frac{\alpha(p)}{p^{s}}\right)^{-1} \left(1 - \frac{\beta(p)}{p^{s}}\right)^{-1} \quad (\Re(s) > 1).
    \]

Associated with each $f(z)$, the $j$-th symmetric power $L$-function is defined by

\begin{align*}
        L(s, \text{sym}^{j}f) &= \sum_{n=1}^{\infty} \frac{\lambda_{\text{sym}^{j}f}(n)}{n^{s}} = \prod_{p} \prod_{m=0}^{j} \left(1 - \frac{\alpha_{f}^{j-m}(p) \beta_{f}^{m}(p)}{p^{s}}\right)^{-1}, \quad \Re(s) > 1.
    \end{align*}

Here, $\textup{sym}^1 f := f$. The moments of $L$-functions play an important role in analytic number theory. In 1927, Ingham \cite{Ingham} established an asymptotic formula for the second moment of the Riemann $\zeta$-function on the critical line using the approximate formula for $\zeta(s)$, namely:
\[\int_{1}^{T} \left\vert\zeta\left(\frac{1}{2} + it\right)\right\vert^{2} \text{d}t = T\log T - (1 + \log 2\pi - 2\gamma)T + O\left(T^{\frac{1}{2} + \varepsilon}\right).\]
In 1979, Heath-Brown \cite{HB} studied the fourth moment of the Riemann $\zeta$-function on the critical line and obtained the following result:
\[\int_{1}^{T} \left\vert\zeta\left(\frac{1}{2} + it\right)\right\vert^{4} \text{d}t = a_4 T\log^4 T + a_3 T\log^3 T + a_2 T\log^2 T + a_1 T\log T + a_0 T + O\left(T^{\frac{7}{8} + \varepsilon}\right),\]
where $a_i$ ($i = 0, 1, \ldots, 4$) are constants. For the higher moments of the $\zeta$-function on the critical line, define $M(A)$ as the infimum of all numbers $M \geq 1$ such that
\[\int_{1}^{T} \left\vert\zeta\left(\frac{1}{2} + it\right)\right\vert^{A} \ll T^{M + \varepsilon}.\]
In 1985, Ivi\'c \cite{Ivic} derived an upper bound for $M(A)$ when $A \geq 4$. Specifically, he showed that
\begin{align*}
        M(A) \leq \begin{cases}
            1 + \frac{A - 4}{8} &\quad 4 \leq A \leq 12,\\
            2 + \frac{3(A - 12)}{22} &\quad 12 \leq A \leq 178/13,\\
            1 + \frac{35(A - 6)}{216} &\quad A \geq 178/13.
        \end{cases}
    \end{align*}
Esipecially, there is
\[
\int_{1}^{T} \left\vert\zeta\left(\frac{1}{2} + it\right)\right\vert^{6} \ll T^{\frac{5}{4} + \varepsilon}.
\]
For $GL(2)$ $L$-function $L(s,f)$, Good \cite{Good} studied the second moment on the critical line and established 
the following asymptotic formula
\[
\int_{1}^{T}\left\vert L\left(\frac{1}{2}+it,f\right)\right\vert^{2}\textup{d}t=2c_{-}T\left(\log\left(\frac{T}{2\pi e}\right)+c_0\right)+O\left(T^{2/3}\log^{2/3}T\right).
\]
It's hard to estimate the higher moments of $L(s,f)$. In 1987 , Jutila \cite{Jutila} considered the sixth moment and obtained the upper bound as following
\[
\int_{1}^{T}\left\vert L\left(\frac{1}{2}+it,f\right)\right\vert^{6}\textup{d}t\ll T^{2+\varepsilon}.
\]
Naturally, the integral moments of $GL(3)$ $L$-function on critical line also are concerned. However, it's even hard 
to get a non-trivial bound of second moment. Denote $L(s,F)$ as the $L$-function associated to Hecke-Maass cusp form
 $F$ for $SL(3,\mathbb{Z})$. In 2022, Pal \cite{Pal} first obtained the non-trivial bound for $L(s,F)$ on critical line. In fact, there is
\[
\int_{T}^{2T}\left\vert L\left(\frac{1}{2}+it,F\right)\right\vert^{2}\text{d}t\ll T^{\frac{3}{2}-\frac{3}{32}+\varepsilon}.
\]
Later, in 2024, Dasgupta et al. \cite{DLY} improved the upper bound
\[
\int_{T}^{2T}\left\vert L\left(\frac{1}{2}+it,F\right)\right\vert^{2}\text{d}t\ll T^{\frac{3}{2}-\frac{1}{6}+\varepsilon}.
\]
In the other hand, the integral moments of $L$-functions for $\frac{1}{2}<\Re s<1$ also are important. In 1985, Ivi\'c \cite{Ivic} invesigated the higher moments of Riemann $\zeta$-function for $\frac{1}{2}<\Re s<1$. In fact, define $m(\sigma)$ as the supremum of all numbers $m$ such that
\[
\int_{1}^{T}\left\vert\zeta(\sigma+it)\right\vert^{m}\text{d}t\ll T^{1+\varepsilon}.
\]
Ivi\'c obtained the lowwer bound of $m(\sigma)$ for $\frac{1}{2}<\sigma<1$. Similarly, Ivi\'c \cite{Ivic0} also considered the moments of $L(\sigma + it, f)$ for $\sigma > \frac{1}{2}$. For $\frac{1}{2} \leq \sigma \leq 1 - \varepsilon$, define $m'(\sigma)$ as the supremum of all numbers $m$ such that
\[\int_{1}^T \vert L(\sigma + it, f)\vert^m \textup{d}t \ll T^{1 + \varepsilon}.\]
 Using the sixth moment of $L(\frac{1}{2} + it, f)$ and an argument similar to that used for the Riemann zeta function, he obtained

\[\quad m'(\sigma) \geq \frac{2}{3 - 4\sigma}, \quad \frac{1}{2} \leq \sigma \leq \frac{5}{8}.\]
Esipecially, for $\sigma=\frac{5}{8}$, there holds
\[
\int_{1}^T \vert L(\frac{5}{8} + it, f)\vert^{4} \textup{d}t \ll T^{1 + \varepsilon}.
\]
In 2016, Liu, Li  Zhang \cite{LLZ} generalized Ivi\'c's result to Maass forms. They proved that

\begin{align*}\label{msigma1}
        m'(\sigma) \geq \frac{5 - 2\sigma}{2(1 - \sigma)(3 - 2\sigma)}, \quad \frac{5}{8} < \sigma \leq 1 - \varepsilon.
    \end{align*}
Now we focus on the symmetric square $L$-function. For $\frac{1}{2} < \sigma < 1$, let $m(\sigma)$ denote the supremum of all numbers $m$ such that
\[\int_{1}^T \vert L(\sigma + it, \textup{sym}^2 f)\vert^m \textup{d}t \ll T^{1 + \varepsilon}.\]
In our previous work \cite{Wang}, we established the bound
\begin{align*}
        m(\sigma) \geq \begin{cases}
            \frac{14}{27 - 30\sigma} & \text{if } \frac{2}{3} \leq \sigma \leq \frac{27}{37},\\
            \frac{5}{6(1 - \sigma)} - \frac{1}{4\sigma} & \text{if } \frac{27}{37} \leq \sigma \leq 1 - \varepsilon.
        \end{cases}
    \end{align*}
Based on recent work on the second moment and subconvexity bound of the symmetric square $L$-function, we can improve the above result for $\frac{5}{8} \leq \sigma \leq \frac{52}{73}$. Specifically, we obtain the following:
	\begin{Th}\label{theorem1}
		Let $L(s, \textup{sym}^2 f)$ be the symmetric square $L$-function related to the
		Hecke eigenform $f$ and $m(\sigma)$ be the supremum of all numbers $m$ such that
		\[
		\int_{1}^T|L(\sigma+it, \textup{sym}^2 f)|^m \textup{d}t\ll T^{1+\varepsilon}.
		\]
		Then we have
		\begin{align*}
			m(\sigma)\geq \frac{17}{26-28\sigma}, \text{ for }\frac{5}{8}\leq\sigma\leq\frac{52}{73}.
		\end{align*}
	\end{Th}
\begin{rem}
Note that the Perelli's \cite{Perelli} classical bound yields $m(\frac{2}{3})=2$, and we improve this bound in the present work..
\end{rem}
\textbf{Notation.} Throughout this paper, the letter $\varepsilon$ represents a sufficiently small positive constant, not necessarily the same at each occurrences. The constants, both explicit and implicit, in Vinogradov symbols may depend on $f$ and $\varepsilon$.
	\section{Preliminary Lemmas}
	In this section, we give some useful lemmas
	which play important roles in the proof of the theorem.
	
	\begin{Lem}\label{short}
		Let $T\leq t \leq 2T$ and $k \geq 1$ be a fixed integer.
		Then for $\frac{1}{2}<\sigma <1$, we have
		\[
		|L(\sigma+ it, \textup{sym}^{2} f)|^k
		\ll 1+\log T\int_{-\log^2 T}^{\log^2 T}
		\left| L\Big(\sigma-\frac{1}{\log T}+i t+ i v, \textup{sym}^2 f\Big)
		\right|^k \exp(-|v|)\textup{d}v.
		\]
		For $\sigma=\frac{1}{2}$, we have
		\[
		\left| L\Big(\frac{1}{2}+ it, \textup{sym}^{2} f\Big)\right|^k
		\ll \log T\left(1+\int_{-\log^2 T}^{\log^2 T} \left|
		L\Big(\frac{1}{2}+i t+ i v, \textup{sym}^2 f\Big)\right| ^k \exp(-\vert v \vert)
		\textup{d}v\right).
		\]
	\end{Lem}
	\begin{proof}
		The proof is similar to \cite[Lemma 7.1]{Ivic}, where we need to use the following function equation which can be found in \cite{Iwaniec}
		\begin{align*}
			\Lambda(s,\textup{sym}^{2}f)=\Lambda(1-s,\textup{sym}^{2}f),
		\end{align*}
	where
	\begin{align*}
		\Lambda(s,\textup{sym}^{2}f)=\pi^{3s/2}\Gamma\left(\frac{s+1}{2}\right)\Gamma\left(\frac{s+k-1}{2}\right)\Gamma\left(\frac{s+k}{2}\right)L(s,\textup{sym}^{2}f).
	\end{align*}
	\end{proof}
	\begin{Lem}\label{intsum}
		For $1/2<\sigma< 1$,
		\[
		\int_{1}^{T}\vert L(\sigma+it, \textup{sym}^2 f)\vert^{m(\sigma)} \textup{d}t
		\ll T^{1+\varepsilon}
		\]
		is equivalent to
		\[
		\sum_{r\leq R}\vert L(\sigma+it_r, \textup{sym}^2 f)\vert^{m(\sigma)}
		\ll T^{1+\varepsilon},
		\]
		where
		\begin{align}\label{tr}
			T\leq t_r\leq 2T \ \text{for}\ r=1,...,R;
			\ \vert t_r-t_s\vert \geq \log^C T\ \text{for}\ 1\leq r\neq s\leq R.
		\end{align}
	\end{Lem}
	\begin{proof}
	The proof is trivial learn from \cite[Sec.8.1, P200]{Ivic}.
	\end{proof}
	\begin{Lem}\label{sumnumber}
		Suppose that $1/2<\sigma<1$ is fixed. Then
		\begin{equation}\label{R}
		R\ll T^{1+\varepsilon}V^{-m(\sigma)}
		\end{equation}
		is equivalent to
		\begin{equation}\label{ER}
		\sum_{r\leq R}\vert L(\sigma+it_r, \textup{sym}^2 f)\vert^{m(\sigma)}
		\ll T^{1+\varepsilon},
		\end{equation}
		where $t_r$ is defined by (\ref{tr}) and
		\[
		\vert L(\sigma+it_r, \textup{sym}^2 f)\vert\geq V\geq T^{\varepsilon}.
		\]
	\end{Lem}
	\begin{proof}
See \cite{Wang}.
	\end{proof}
	
	\begin{Lem}\label{Dirichletvalue}
		Let $t_1<\cdots<t_R$ be real numbers such that $T\leq t_r\leq 2T$ for
		$r=1,2,\cdots, R$ and $\vert t_r-t_s\vert \geq \log^4 T$ for
		$1\leq r\neq s\leq R$. If
		\begin{align*}
			T^{\varepsilon}<V\leq\left\vert\sum_{M\leq n\leq 2M}a(n)n^{\sigma-it_r}\right\vert,
		\end{align*}
		where $a(n)\ll M^\varepsilon$ for $M\leq n\leq 2M$, $1\ll M\ll T^C(C>0)$. Then
		\[
		R\ll T^\varepsilon\left(M^{2-2\sigma}V^{-2}+TV^{-f(\sigma)}\right).
		\]
		Here
		\begin{align*}
			f(\sigma)=
			\begin{cases}
				\frac{2}{3-4\sigma} &\textup{if}\ \frac 12<\sigma\leq \frac 23,\\
				\frac{10}{7-8\sigma} &\textup{if}\ \frac 23<\sigma\leq \frac{11}{14},\\
				\frac{34}{15-16\sigma} &\textup{if}\ \frac{11}{14}<\sigma\leq \frac{13}{15},\\
				\frac{98}{31-32\sigma} &\textup{if}\ \frac{13}{15}<\sigma\leq \frac{57}{62},\\
				\frac{5}{1-\sigma} &\textup{if}\ \frac{57}{62}<\sigma\leq 1-\varepsilon.
			\end{cases}
		\end{align*}
	\end{Lem}
	\begin{proof}
		See Ivi\'{c} \cite[Lemma 8.2]{Ivic} for details.
	\end{proof}
	
\begin{Lem}\label{sym2sub}
		For any $\varepsilon>0$, then
\begin{equation*}
\int_{-T}^{T}\left\vert L\left(\frac{1}{2}+it,\textup{sym}^2f\right)\right\vert^2\textup{d}t\ll T^{\frac{4}{3}+\varepsilon},
\end{equation*}
and for $\frac{1}{2}\leq\sigma\leq1+\varepsilon$, 
		\begin{equation*}
			L(\sigma+it,\textup{sym}^{2}f)\ll(1+\vert t\vert)^{\max\{\frac{8}{7}(1-\sigma),0\}+\varepsilon}.
		\end{equation*}
	\end{Lem}
	\begin{proof}
		See Dasguta et al. \cite{DLY}.
	\end{proof}	
	\section{Proof of Theorem \ref{theorem1}}
	
	By Lemma \ref{intsum} and Lemma \ref{sumnumber}, it suffices to prove
	\begin{equation}\label{eq1}
		R\ll T^{1+\varepsilon}V^{-m(\sigma)}.
	\end{equation}
	Standardly, Mellin's transform implies that
	\[
	\sum_{n\geq 1} \lambda_{\text{sym}^2 f}(n) e^{-n/Y}n^{-s}
	=\frac{1}{2\pi i}\int_{(2)}Y^{w}\Gamma(w)L(s+w, \text{sym}^2 f)\textup{d}w,
	\]
	where $Y$ is a parameter which will be chosen later and  $s=\sigma+it_r$ with
	$\frac{2}{3}\leq \sigma <1$ and $t_r$ satisfying the condition
	of Lemma \ref{intsum}.
	Moving the integral line to $\Re w=\frac{1}{2}-\sigma$, by the Cauchy's residue theorem,
	we get
	\begin{align*}
		\sum_{n\leq Y} \lambda_{\text{sym}^2 f}(n) e^{-n/Y}n^{-s}
		=L(s, \text{sym}^2 f)+\frac{1}{2\pi i}\int_{(\frac{1}{2}-\sigma)}
		Y^{w}\Gamma(w)L(s+w, \text{sym}^2 f)\textup{d}w.
	\end{align*}
	For $T \to \infty$, Stirling's formula implies that
	\begin{align*}
		L(s, \text{sym}^2 f)\ll &\left\vert \sum_{n\leq Y}\lambda_{\text{sym}^2 f}(n)
		e^{-\frac{n}{Y}}n^{-s}\right\vert \\
		&+\left|\int_{-\log^2 T}^{\log^2 T}Y^{(\frac{1}{2}-\sigma)}\Gamma
		\left(\left(\frac{1}{2}-\sigma\right)+iv\right)
		L\left(\frac{1}{2}+it_r+iv, \text{sym}^2 f\right)\textup{d}v\right|,
	\end{align*}
	which implies that either
	\begin{align*}
		V\ll \log T \max_{M\leq Y/2}\left\vert \sum_{M\leq n\leq 2M}
		\lambda_{\text{sym}^2 f}(n) e^{-\frac{n}{Y}}n^{-s}\right\vert
	\end{align*}
	or
	\begin{align*}
		V\ll Y^{\left(\frac{1}{2}-\sigma\right)}\log^2 T
		\left\vert L\left(\frac{1}{2}+it_r^\prime, \text{sym}^2 f\right)\right\vert.
	\end{align*}
	Here
	\[
	\left\vert L\left(\frac{1}{2}+it_r^\prime, \text{sym}^2 f\right)\right\vert
	=\max_{-\log^2 T\leq v\leq \log^2 T}\left\vert
	L\left(\frac{1}{2}+it_r+iv, \text{sym}^2 f\right)\right\vert.
	\]
It's easy to verify that the conditions of  Lemma \ref{intsum} and \ref{sumnumber}, according to the defintion of $t_r^{\prime}$ and the analytic properties of $L(s,\textup{sym}^{2}f)$. Therefore, it suffices to prove 
	\begin{equation*}
		R\ll T^{1+\varepsilon}V^{-m(\sigma)}.
	\end{equation*}
For convenice, denote $z:=m(\sigma)$, we next prove that
\begin{equation*}
		R\ll T^{1+\varepsilon}V^{-z}.
	\end{equation*}
By Lemma \ref{Dirichletvalue}, we have
\[
	R\ll T^{\varepsilon}\left(Y_1^{2-2\sigma}V^{-2}
	+TV^{-f(\sigma)}\right)+Y^{\frac{1}{2}-\sigma}V^{-1}
	\sum_{r\leq R_1}\left\vert L\left(\frac{1}{2}+it_r^\prime, \text{sym}^2 f\right)\right\vert.
	\]
Then we need separate two cases to bound $R$.\\
Case 1: $\frac{5}{8}\leq \sigma\leq\frac{2}{3}$.\\
At this position, $f(\sigma)=\frac{2}{3-4\sigma}$, then
\begin{equation}\label{wholeR}
R\ll T^{\epsilon}\left(Y_1^{2-2\sigma}V^{-2}
	+TV^{-\frac{2}{3-4\sigma}}\right)+Y^{\frac{1}{2}-\sigma}V^{-1}
	\sum_{r\leq R}\left\vert L\left(\frac{1}{2}+it_r^\prime, \text{sym}^2 f\right)\right\vert.
\end{equation}
By Lemma \ref{short} and along with Lemma \ref{sym2sub}, we get
	\begin{align}\label{power2}
		\sum_{r\leq R}\left\vert
		L\left(\frac{1}{2}+it_r^\prime, \text{sym}^2 f\right)\right\vert^{2}
		&\ll\sum_{r\leq R}\log T\left(1+\int_{-\log^{2}T}^{\log^{2}T}\left\vert L\left(\frac{1}{2}+it_{r}+iv,\text{sym}^{2}f\right)\right\vert^{2}\exp(-\vert v\vert)\right)\text{d}v \notag\\
&\ll R_1T^{\varepsilon}+\log T\int_{T}^{2T+\log^2 T} \left\vert
		L\left(\frac{1}{2}+i t, \text{sym}^2 f\right)\right\vert ^2 \textup{d}t \notag \\
		&\ll T^{\frac{4}{3}+\varepsilon}.
	\end{align}
From Cauchy inequation and \eqref{power2}, we obtain
	\[
	R\ll T^{\varepsilon}\left(Y_1^{2-2\sigma}V^{-2}
	+TV^{-\frac{2}{3-4\sigma}}\right)+Y^{1-2\sigma}V^{-2}T^{\frac{4}{3}+\varepsilon}.
	\]
Taking $Y=T^{\frac{4}{3}}$, then
\[
R\ll T^{1+\varepsilon}V^{-\frac{2}{3-4\sigma}}+T^{\frac{8}{3}-\frac{8}{3}\sigma+\varepsilon}V^{-2}.
\]
Set
\[
T^{1+\varepsilon}V^{-\frac{2}{3-4\sigma}}\ll T^{1+\varepsilon}V^{-z},
\]
we get
\[
z\leq\frac{2}{3-4\sigma}.
\]
Let
\[
T^{\frac{8}{3}-\frac{8}{3}\sigma+\varepsilon}V^{-2}\ll T^{1+\varepsilon}V^{-z}.
\]
Now we need consider for $z<2$ and $z\geq2$. For the first case, if $z<2$, we have
\[
V\gg T^{(\frac{5}{3}-\frac{8}{3}\sigma)\cdot\frac{1}{2-z}}.
\]
Since $\frac{5}{8}\leq \sigma\leq \frac{2}{3}$, we get
\[
V\gg T^{(\frac{5}{3}-\frac{8}{3}\sigma)\cdot\frac{1}{2-z}}\gg T^{-\varepsilon}.
\]
And  notice that $V\gg T^{\varepsilon}$, thus the above always holds for $z<2$. Similarly for $z=2$. Next arrive at $z>2$.
In this position, we get
\[
V\ll T^{(\frac{5}{3}-\frac{8}{3}\sigma)\cdot\frac{1}{2-z}}.
\]
For 
\[
V\gg T^{(\frac{5}{3}-\frac{8}{3}\sigma)\cdot\frac{1}{2-z}},
\]
by H\"{o}lder inequality,  we have
	\[
	R\ll Y^{2-2\sigma}V^{-2}+TV^{-\frac{2}{3-4\sigma}}
	+Y^{\frac{1}{2}-\sigma}V^{-1}
	R^{\frac{3}{4}}\left(\sum_{r\leq R}\left\vert
	L\left(\frac{1}{2}+it_r^{\prime}, \text{sym}^{2}f\right)\right\vert^4\right)
	^{\frac{1}{4}}.
	\]
According to Lemma \ref{sym2sub} and \eqref{power2},then
\[
R\ll Y^{2-2\sigma}V^{-2}T^{\varepsilon}+T^{1+\varepsilon}V^{-\frac{2}{3-4\sigma}}+Y^{2-4\sigma}V^{-4}T^{\frac{52}{21}+\varepsilon}.
\]
Taking $Y=T^{\frac{26}{21\sigma}}V^{-\frac{1}{\sigma}}$, we have
\[
R\ll T^{\frac{52(1-\sigma)}{21\sigma}+\varepsilon}V^{-\frac{2}{\sigma}}+T^{1+\varepsilon}V^{-\frac{2}{3-4\sigma}}.
\]
Let
\[
T^{\frac{52(1-\sigma)}{21\sigma}+\varepsilon}V^{-\frac{2}{\sigma}}\ll T^{1+\varepsilon}V^{-z}.
\]
Here we can suppose $z<\frac{2}{\sigma}$. In fact, for $z>\frac{2}{\sigma}$, we can essentially get a contradict on the range of $V$. Hence
\[
V\gg T^{(\frac{52}{21\sigma}-\frac{73}{21})\cdot\frac{1}{\frac{2}{\sigma}-z}}.
\]
Conbining the inequality about $z$, we get an inequality system
\begin{align*}
\begin{cases}
(\frac{5}{3}-\frac{8}{3}\sigma)\cdot\frac{1}{2-z}\leq(\frac{52}{21\sigma}-\frac{73}{21})\cdot\frac{1}{\frac{2}{\sigma}-z},\\
z\leq\frac{2}{3-4\sigma},\\
z>2,\\
z<\frac{2}{\sigma}.
\end{cases}
\end{align*}
Then we get
\[
z\geq\frac{17}{26-28\sigma}.
\]
Notice that for $\sigma=\frac{5}{8}$, $\frac{34}{52-56\sigma}=2$. Thus for $\frac{5}{8}\leq\sigma\leq\frac{2}{3}$, $m(\sigma)\geq\frac{17}{26-28\sigma}$.\\
Case 2 : $\frac{2}{3}\leq \sigma \leq \frac{52}{73}$.\\
In this case, suppose $z\geq\frac{17}{26-28\sigma}$, and we essentially to prove it. And notice that $f(\sigma)=\frac{10}{7-8\sigma}$, we get
\begin{equation*}
R\ll T^{\epsilon}\left(Y_1^{2-2\sigma}V^{-2}
	+TV^{-\frac{10}{7-8\sigma}}\right)+Y^{\frac{1}{2}-\sigma}V^{-1}
	\sum_{r\leq R}\left\vert L\left(\frac{1}{2}+it_r^\prime, \text{sym}^2 f\right)\right\vert.
\end{equation*}
For convenice, by Case 1, according to $V\leq T^{\frac{26-28\sigma}{21}}$ or not, we separate \{$t_r$\}into two parts $S_1$, $S_2$. For $S_1$, by Cauchy inequality and \eqref{power2}, we obtain
	\[
	R_1\ll T^{\varepsilon}\left(Y_1^{2-2\sigma}V^{-2}
	+TV^{-\frac{10}{7-8\sigma}}\right)+Y^{1-2\sigma}V^{-2}T^{\frac{4}{3}+\epsilon}.
	\]
Taking $Y_1=T^{\frac{4}{3}}$, then
	\[
	R_1\ll T^{\frac{8}{3}-\frac{8}{3}\sigma+\varepsilon}V^{-2}+T^{1+\varepsilon}V^{-\frac{10}{7-8\sigma}}
	.
	\]
By $V\leq T^{\frac{26-28\sigma}{21}}$, we get
\[
R_1\ll T^{1+\varepsilon}V^{-\frac{17}{26-28\sigma}}.
\]
To estimate $R_2$, we apply H\"{o}lder inequality, then
	\[
	R_2\ll Y_2^{2-2\sigma}V^{-2}+TV^{-\frac{10}{7-8\sigma}}
	+Y_{2}^{\frac{1}{2}-\sigma}V^{-1}
	R_2^{\frac{3}{4}}\left(\sum_{r\leq R}\left\vert
	L\left(\frac{1}{2}+it_r^{\prime}, \text{sym}^{2}f\right)\right\vert^4\right)
	^{\frac{1}{4}},
	\]
where we have used the subconvexity bound of $L(s, \text{sym}^2 f)$ and \eqref{power2}. Simplifying it, we have
	\[
	R_2\ll Y^{2-2\sigma}V^{-2}T^{\varepsilon}+T^{1+\varepsilon}V^{-\frac{2}{3-4\sigma}}+Y^{2-4\sigma}V^{-4}T^{\frac{52}{21}+\varepsilon}.
	\]
Taking $Y_2=T^{\frac{26}{21\sigma}}V^{-\frac{1}{\sigma}}$, then 
\[
R_2\ll T^{\frac{52(1-\sigma)}{21\sigma}+\varepsilon}V^{-\frac{2}{\sigma}}+T^{1+\varepsilon}V^{-\frac{10}{7-8\sigma}},
\] 
along with
\[
V\geq T^{\frac{26-28\sigma}{21}},
\]
hence
	\[
	R_2\ll T^{1+\varepsilon}V^{-\frac{17}{26-28\sigma}}.
	\]
Therefore, we have
	\[
	R=R_1+R_2\ll T^{1+\varepsilon}V^{-\frac{17}{26-28\sigma}}.
	\]
It completes the proof.

\halign{\small\qquad\quad#\hfill\qquad&\small#\hfill\qquad&\small#\hfill \qquad &
	\small#\hfill \cr
	Youjun Wang    \cr                                          
	School of Mathematics and Statistics, Henan University, Kaifeng, Henan 475004,                                    
	P. R. China                                               \cr
	\texttt{math$\_$wyj@henu.edu.cn}                               \cr
}
\end{document}